\documentclass[review]{elsarticle}
\usepackage{lineno,hyperref}
\modulolinenumbers[5]
\journal{Applied Numerical Mathematics}
\usepackage [latin1]{inputenc}
\usepackage{latexsym}
\usepackage{amssymb}
\usepackage{amsmath}
\usepackage{ntheorem}
\usepackage[left=1in,right=1in,top=1in,bottom=1in]{geometry}
\usepackage{graphicx}
\usepackage{subfigure}
\usepackage{epstopdf}
\usepackage{enumerate}
\usepackage{longtable}
\usepackage{bm}
\usepackage{caption}
\usepackage{algorithm}
\usepackage{algorithmicx}
\usepackage{algpseudocode}
\usepackage{array}
\usepackage{times}
\usepackage{mathptmx}
\usepackage{multirow}
\usepackage[misc]{ifsym}
\usepackage{natbib}
\usepackage{booktabs}
\usepackage{multirow}
\usepackage{threeparttable}
\usepackage{paralist}

\newtheorem{thm}{Theorem}[section]

\newtheorem{lem}{Lemma}[section]
\newtheorem{proof}{Proof}

\newtheorem{exm}{Example}[section]
\numberwithin{equation}{section}
\numberwithin{figure}{section}
\numberwithin{algorithm}{section}
\begin{document}
\begin{frontmatter}
\title{On Kaczmarz method with oblique projection for solving large overdetermined linear systems \tnoteref{mytitlenote}}
\tnotetext[mytitlenote]{This research is supported by National Key Research and Development Program of China (grant number 2019YFC1408400), the Fundamental Research Funds for the Central Universities (grant number 19CX05003A-2) and the Fundamental Research Funds for the Central Universities (grant number 18CX02041A)}
\author[mymainaddress]{Weiguo Li}
\author[mymainaddress]{Qifeng Wang}
\author[mymainaddress]{Wendi Bao}
\cortext[mycorrespondingauthor]{Corresponding author}
\ead{baowendi@sina.com}
\author[mymainaddress]{Li Liu}
\address[mymainaddress]{College of Science, China University of Petroleum, Qingdao 266580, P .R. China}
\begin{abstract}
In this paper, an extension of Kaczmarz method, the Kaczmarz method with oblique projection (KO), is introduced and analyzed. Using this method, a number of iteration steps to solve the over-determined systems of linear equations are significantly reduced, and the the computing time is much saved, especially for those problems that contain some linear equations with near-linear correlation. Simultaneously, a randomized version--randomized Kaczmarz method with oblique projection (RKO) is established. The convergence proofs of these two methods are given and numerical experiments show the effectiveness of the two methods for uniformly distributed random data. Especially when the system has correlated
rows, the improvement of experimental results is very prominent.
\end{abstract}

\end{frontmatter}

\section{Introduction}
\label{intro}
Consider solving an overdetermined consistent linear system of equations
\begin{align}\label{1.1}
Ax=b,
\end{align}
where the matrix $A\in \mathbb{R}{^{m \times n}}(m \ge n)$, $b\in\mathbb{R}{^m}$. Denote the rows of $A$ by $a_1^T$,$a_2^T$, ..., $a_m^T$ and let $b = \left( {{b_1},{b_2},...,{b_m}} \right)^T$. The Kaczmarz method \cite{K37} (or the algebraic reconstruction technique (ART) \cite{GBH70}) is one of the most popular solvers. At each iteration, the Kaczmarz method uses the cyclic rule to choose a row of the matrix and projects the current iteration onto the corresponding hyperplane.
The convergence rates for Kaczmarz-type algorithms are given by C. Popa \cite{P17}. In 2009, Strohmer and Vershynin \cite{KV09} proved firstly that the randomized Kaczmarz (RK) method converges with the linear rate when the rows are selected randomly with probability proportional to the square of the Euclidean norm of the rows.

As far back as 1954 \cite{A54}, the linear convergence of a greedy projection method, known as Motzkin's method (or the relaxation method) \cite{MS54}, was proved. This method was also called Kaczmarz method with the `maximal-residual control' in the numerical linear algebra literatures \cite{C81,NSV16}. A discussion on the full history of this method's linear convergence results and this greedy projection method is given in \cite{C81} or \cite{LHN17}.
There are at least two greedy selection rules: the maximum residual (MR) rule and the maximum distance (MD) rule, respectively:
\begin{equation}
i_k = \arg\max\limits_{i}|\langle a_{i}, x\rangle-b_i| \ \ (MR); \ \ \ \  i_k = \arg\max\limits_{i}|\langle a_{i}, x\rangle-b_i|/\|a_i\|_2 \ \ (MD)
\end{equation}
where $i_k$ is the row index that should be selected at the $kth$ iteration.

With the MR rule, the proof of linear convergence rate was provided by R. Ansorge in \cite{A84}. For further developments of Ansorge's Maximal residual algorithm, we can refer to \cite{PP16} and \cite{P17}. Recently, a series of effective probability criteria for selecting the working rows from the coefficient matrix are introduced. For example, the greedy randomized Kaczmarz (GRK) method \cite{BW18SISC,BW19NLAA}, the new randomized Kaczmarz (NRK) method \cite{GL20} are constructed. These methods converge to the unique least-norm solution of the linear system when it is consistent. Theoretical analysis demonstrates that the convergence rates of the GRK and NRK methods are much faster than that of the randomized Kaczmarz method \cite{BW18SISC,BW19NLAA,GL20}.

In the above methods, each iteration is an orthogonal projection. For coherent over-determined systems, the speed of iterative improvement is very slow. So it is very meaningful to introduce an accelerate projection to improve the convergence rate of this type of equations. In this paper, we introduce and analyze an oblique projection method to improve the convergence rate of the original Kaczmarz method.

 Here is a short outline of the paper: in Section 2, we introduce a new method--Kaczmarz method with oblique projection (KO) and present its algorithm and its convergence proof. Then, an randomized version of KO method (RKO) for coherent overdetermined systems and its proof of convergence rate are provided in Section 3. In Section 4, we present a variety of numerical experiments for over-determined systems of linear equations with uniformly distributed random data. The final section is devoted to some remarks and conclusions.

\section{Kaczmarz Method with Oblique Projection}
\subsection{Kaczmarz Method with Orthogonal Projection}

We use the following notations. $\|x\|$ is the Euclid norm of $x\in R^n$, $\|A\|=\max\limits_{\|x\|=1}\|Ax\|$ for $A\in R^{m\times n}$; $A^{\dag}$ is the Moore-Penrose inverse of $A$;
$A^T$ is the transpose of $A$; $R(A)$ is the range of the matrix $A$; $N(A)$ is the null space of the matrix $A$;
$P_C(x)$ is the orthogonal projection of $x$ onto $C$; $\sigma_{min}(A)$ is the smallest nonzero singular value of $A$;
$\tilde{x}$ is a solution of \eqref{1.1}; $x^*=A^{\dag}b$ is the least-norm solution of \eqref{1.1}.

The Kaczmarz algorithm with orthogonal projection is described as follows.

{\bf (Cyclic) Kaczmarz Algorithm

Initialization: $x^{(0)}\in R^n$,

Iterative step: for $k = 0, 1, \cdots$ select $i_{k+1} =k\ (mod\ m) + 1$ and compute $x^{(k+1)}$ as
\begin{equation}
x^{(k+1)}=x^{(k)}+\frac{b_{i_{k+1}}-\langle a_{i_{k+1}}, x^{(k)}\rangle}{\|a_{i_{k+1}}\|^2}a_{i_{k+1}}.
\end{equation}}

The next result came from Theorem 4 in \cite{P17} about the convergence rate of this algorithm.

\begin{lem}[\cite{P17}]
 \ Let $x^{(0)}\in R^n$ be an arbitrary initial approximation, $\tilde{x}$ is a solution of \eqref{1.1}
such that $P_{N(A)}(\tilde{x})= P_{N(A)}(x^{(0)})$, and the sequence $\{x^{(k)}\}$ is generated with this Kaczmarz
algorithm. Then, there exists a constant $\delta\in[0,1)$ such that
\begin{equation}\|x^{(k)}-\tilde{x}\|_2\leq\delta^{m_k},\end{equation}
where $m_k$ and $q_k\in\{0, 1,\cdots,m-1\}$ are (uniquely) defined by $k=m\cdot m_k+q_k$.
\end{lem}

In fact, the Kaczmarz method has the following convergence properties (see Theorem 1 in \cite{EHL81}).

\begin{lem}[\cite{EHL81}]
 \ If \eqref{1.1} is consistent, then the sequence $\{x^{(k)}\}$ generated with the Kaczmarz
algorithm converges to a solution of \eqref{1.1}. If, in addition, $x^{(0)}\in R(A^T)$, then $\{x^{(k)}\}$ converges to
the least-norm solution of \eqref{1.1}, i.e.,
\begin{equation*}\lim\limits_{k\rightarrow\infty}x^{(k)}=x^*.\end{equation*}
\end{lem}

\subsection{Kaczmarz Method with Oblique Projection}
Consider Kaczmarz method with oblique projection. A question is raised: if we have computed $x^{(k)}$, how to find the next iteration $x^{(k+1)}$ which is on the intersection of two hyperplanes? Here we give a simple strategy to find a next iteration $x^{(k+1)}$ that can converge to $\tilde{x}$ much quickly, where $\tilde{x}$ is a solution of the linear system \eqref{1.1}.

Our Kaczmarz method with oblique projection is described as follows (refer to Fig. 2.1):

Assume that $x^{(k)}$ is the $k$th iteration of solving the systems of equations \eqref{1.1}, and $x^{(k)}$ is on the hyperplane $\langle a_{i_k}, x\rangle=b_{i_k}$. Orthogonal projection from point $x^{(k)}$ to the hyperplane $\langle a_{i_{k+1}}, x\rangle=b_{i_{k+1}}$, and get the projection point $y^{(k)}$. And then orthogonal projection from point $y^{(k)}$ to the hyperplane $\langle a_{i_k}, x\rangle=b_{i_k}$, and get the projection point $z^{(k)}$. Let line L pass through point $x^{(k)}$ and along direction $w^{(i_k)}=z^{(k)}-x^{(k)}$, then the intersection of L and the hyperplane $\langle a_{i_{k+1}}, x\rangle=b_{i_{k+1}}$ is chosen as the next iteration point $x^{(k+1)}$ (denote $x^{(k+1)}$ as the oblique projection point of $x^{(k)}$ along $w^{(i_k)}$ to the hyperplane $\langle a_{i_{k+1}}, x\rangle=b_{i_{k+1}}$). In the following lemma, we will deduce the iterative formula:
\begin{equation}\label{itefor}
x^{(k+1)}=x^{(k)}+t_{k}w^{(i_k)},
\end{equation}
where $t_{k}$ is step size.
\begin{figure}[htpb]\label{fig:KO}
  \centering
  \includegraphics[width=0.6\textwidth]{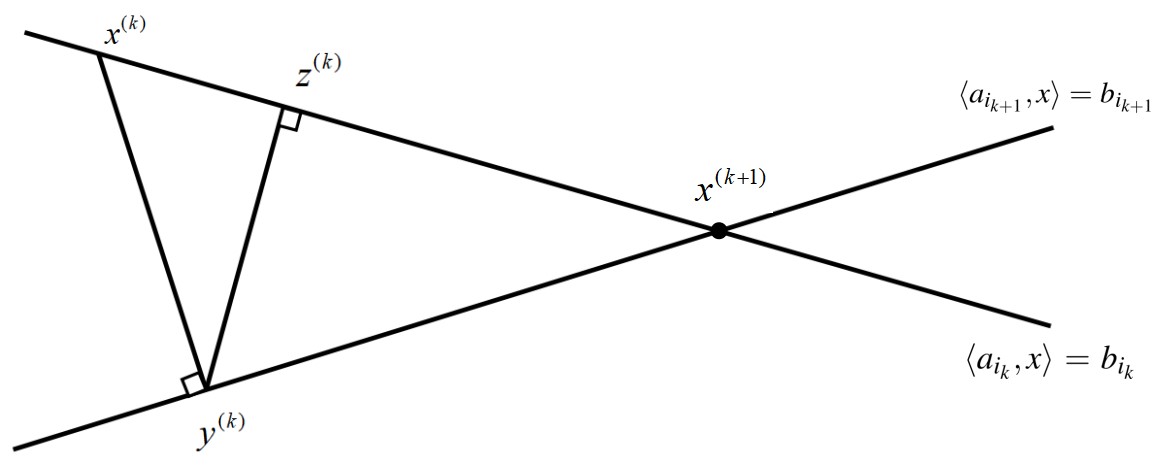}
\caption{KO, m=2}
\end{figure}

\begin{figure}\label{KO2}
\centering
\includegraphics[width=0.6\textwidth]{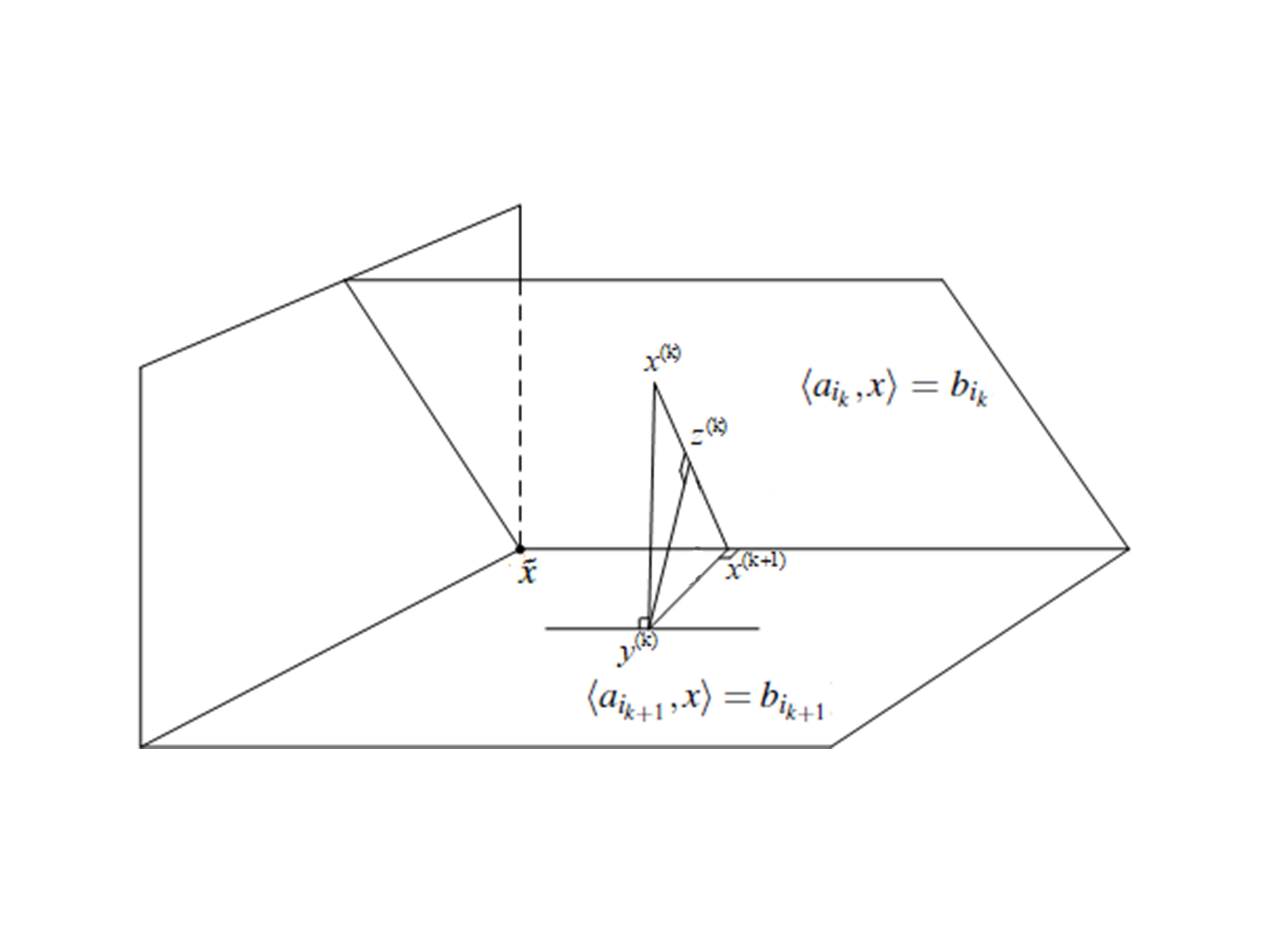}
\setlength{\abovecaptionskip}{0cm}
\setlength{\belowcaptionskip}{0pt}
\begin{center}\caption{KO,m=3}\end{center}
\end{figure}
\begin{lem}\label{lem2.3}
Let the step size $t_{k}$, the direction vector $w^{(i_k)}$,  $x^{(k)}$ and $x^{(k+1)}$ be the same as the definitions of the iterative scheme \eqref{itefor}, then in the KO method, we have
$$t_k=\frac{b_{i_{k+1}}-\langle a_{i_{k+1}}, x^{(k)}\rangle}{\langle a_{i_{k+1}}, w^{(i_k)}\rangle},
\ w^{(i_k)}={a_{i_{k+1}}}-\frac{\langle a_{i_{k}}, a_{i_{k+1}}\rangle}{\langle a_{i_k}, a_{i_{k}}\rangle}a_{i_k}.$$
\end{lem}
\begin{proof}
According to the above description of Kaczmarz method with oblique projection and Figure 2.1 and 2.2,
\begin{align*}
y^{(k)}=x^{(k)}+\frac{b_{i_{k+1}}-\langle a_{i_{k+1}},x^{(k)}\rangle}{\|a_{i_{k+1}}\|^2}a_{i_{k+1}},
\end{align*}
and
\begin{align*}
z^{(k)}=y^{(k)}+\frac{b_{i_{k}}-\langle a_{i_{k}},y^{(k)}\rangle}{\|a_{i_{k}}\|^2}a_{i_{k}}.
\end{align*}
Hence,
\begin{align*}
w^{(i_{k})}&=z^{(k)}-x^{(k)}\\
&=\frac{b_{i_{k+1}}-\langle a_{i_{k+1}},x^{(k)}\rangle}{\|a_{i_{k+1}}\|^2}a_{i_{k+1}}+\frac{b_{i_{k}}-\langle a_{i_{k}},y^{(k)}\rangle}{\|a_{i_{k}}\|^2}a_{i_{k}}\\
&=\lambda({a_{i_{k+1}}}+\mu{a_{i_{k}}}).
\end{align*}
For convenience, the factor $\lambda$ can be omitted, which does not affect the change of the unit direction of $w^{(i_{k})}$.
According to the orthogonality of $w^{(i_k)}$ and ${a_{i_{k}}}$, we get
$\mu=-\frac{\langle a_{i_{k}}, a_{i_{k+1}}\rangle}{\langle a_{i_k}, a_{i_{k}}\rangle}$, then
\begin{equation*}
w^{(i_k)}={a_{i_{k+1}}}-\frac{\langle a_{i_{k}}, a_{i_{k+1}}\rangle}{\langle a_{i_k}, a_{i_{k}}\rangle}a_{i_k}.
\end{equation*}
Taking an inner product on $a_{i_{k+1}}$ with both sides of the equation \eqref{itefor} and subtracting $b_{i_{k+1}}$, we get
\begin{align}\label{aik+1}
0=\langle a_{i_{k+1}},x^{(k)}\rangle-b_{i_{k+1}}+t_{k}\langle a_{i_{k+1}},w^{(i_{k})}\rangle,
\end{align}
so
\begin{align*}
t_k=\frac{b_{i_{k+1}}-\langle a_{i_{k+1}}, x^{(k)}\rangle}{\langle a_{i_{k+1}}, w^{(i_k)}\rangle}.
\end{align*}
The reason why the left equation of \eqref{aik+1} is equal to $0$ is that $x^{(k+1)}$ is on the hyperplane $\langle a_{i_{k+1}},x\rangle=b_{i_{k+1}}$, i.e. $\langle a_{i_{k+1}},x^{(k+1)}\rangle=b_{i_{k+1}}$.
\end{proof}
\begin{lem}\label{lem2.4}
Assume that $\tilde{x}$ is a solution of the linear system \eqref{1.1}, $w^{(i_k)}$ is the direction vector from $x^{(k)}$ to $x^{(k+1)}$ in the KO method. Then $w^{(i_k)}$ and $x^{(k+1)}-\tilde{x}$ are orthogonal, i.e., $\langle w^{(i_k)},x^{(k+1)}-\tilde{x} \rangle=0$.
\end{lem}
\begin{proof}
 On the basis of the description of the KO algorithm ( for the three dimensions, we can see Fig. 2.2 ), $x^{(k+1)}$ is the oblique projection point of $x^{(k)}$ along $w^{(i_{k})}$ to the hyperplane $\langle a_{i+1},x\rangle =b_{i+1}$. So the points $x^{(k+1)}$ and $\tilde{x}$ are on the hyperplane $\langle a_{i+1},x\rangle =b_{i+1}$, then $\langle a_{i+1},x^{(k+1)}-\tilde{x}\rangle =0$. In addition, the points $x^{(k+1)}$ and $\tilde{x}$ are also on the hyperplane $\langle a_{i},x\rangle =b_{i}$, therefore, $\langle a_{i},x^{(k+1)}-\tilde{x}\rangle =0$. According to the definition of $w^{(i_{k})}$, $\langle w^{(i_{k})}, x^{(k+1)}-\tilde{x}\rangle=0$.
\end{proof}

With Lemma \ref{lem2.3}, the algorithm is described as in Algorithm 2.1. Without losing generality, we assume that all rows of $A$ are not zero vectors.

\begin{algorithm}
  \leftline{\caption{Kaczmarz Method with Oblique Projection (KO)}}
  \label{alg:Framwork1}
  \begin{algorithmic}[1]
    \Require
      $A\in R^{m\times n}$, $b\in R^{m}$, $x^{(0)}\in R^n$, $K$, $\varepsilon>0$
    \State For $i=1:m$, $M(i)=\|a_i\|^2$
    \State Compute $x^{(1)}=x^{(0)}+\frac{b_1-\langle a_1, x^{(0)}\rangle}{M(1)}a_1$ and set $i_{k+1}=1$
    \For {$k=1,2,\cdots, K$}
    \State Set $i_k=i_{k+1}$ and choose a new $i_{k+1}$: $i_{k+1}=mod(k,m)+1$
    \State Compute $D_{i_k}=\langle a_{i_{k}}, a_{i_{k+1}}\rangle$ \ and $r_{i_{k+1}}^{(k)}=b_{i_{k+1}}-\langle a_{i_{k+1}}, x^{(k)}\rangle$
    \State Compute $w^{(i_k)}=a_{i_{k+1}}-\frac{D_{i_k}}{M(i_k)}a_{i_{k}}$ \ and \ $h_{i_k}(=\|w^{(i_k)}\|^2)=M(i_{k+1})-\frac{D_{i_k}}{M(i_k)}D_{i_k}$
    \If {$h_{i_k}>\varepsilon$}
    \State $\alpha_{i_k}^{(k)}=\frac{r_{i_{k+1}}^{(k)}}{h_{i_k}}$ \ and $x^{(k+1)}=x^{(k)}+\alpha_{i_k}^{(k)} w^{({i_k})}$
    \EndIf
    \EndFor
    \State Output $x^{(k+1)}$
  \end{algorithmic}
\end{algorithm}

Assume the system \eqref{1.1} is consistent, then it must be $b_{i_k}=\lambda b_{i_{k+1}}$ if the two rows of the coefficient matrix $A$ have relation $a_{i_k}=\lambda a_{i_{k+1}}$. In this case, the two hyperplanes $\langle a_{i_k},x\rangle=b_{i_k}$ and $\langle a_{i_{k+1}}, x\rangle=b_{i_{k+1}}$  are coincident, and we can eliminate one of them without affecting the solution of the equations. So in the following proof, we always assume that $0<\theta_{i_{k+1}}\leq\pi/2$, here $\theta_{i_{k+1}}$ is the angle between any two hyperplanes $\langle a_{i_k},x\rangle=b_{i_k}$ and $\langle a_{i_{k+1}}, x\rangle=b_{i_{k+1}}$. In the Algorithm 2.1, $$h_{i_k}=\|w^{(i_k)}\|^2=\frac{1}{\|a_{i_k}\|^2}\|a_{i_{k+1}}\|^2\|a_{i_k}\|^2\left(1-cos^2(\theta_{i_k})\right)=\|a_{i_{k+1}}\|^2sin^2(\theta_{i_{k+1}}),$$
thus $h_{i_k}\geq\epsilon>0$ ensures $\theta_{i_{k+1}}>0$ because $h_{i_k}=\|a_{i_{k+1}}\|^2sin^2(\theta_{i_{k+1}})$.

Before giving the proof of the convergence of the KO algorithm, we first restate the KO algorithm as the following process.
For $x^{(0)}\in R^{n}$ as an initial approximation we define $x^{(0,0)},x^{(0,1)},\cdots ,x^{(0,m)}\in R^{n}$ by
\begin{equation}\label{eq2.10}
\left\{
             \begin{array}{lr}
             x^{(0,0)}=x^{(0)}+\frac{b_1-\langle a_1,x^{(0)}\rangle}{\|a_1\|^2}a_1, &  \\
             x^{(0,1)}=x^{(0,0)}+\frac{b_2-\langle a_2,x^{(0,0)}\rangle}{\|w^{(1)}\|^2}w^{(1)},&\\
             x^{(0,2)}=x^{(0,1)}+\frac{b_3-\langle a_3,x^{(0,1)}\rangle}{\|w^{(2)}\|^2}w^{(2)},&\\
             \cdots \cdots \cdots \cdots \cdots \cdots \cdots \cdots \cdots &\\
             x^{(0,m-1)}=x^{(0,m-2)}+\frac{b_m-\langle a_m,x^{(0,m-2)}\rangle}{\|w^{(m-1)}\|^2}w^{(m-1)},&\\
             x^{(0,m)}=x^{(0,m-1)}+\frac{b_1-\langle a_1,x^{(0,m-1)}\rangle}{\|w^{(m)}\|^2}w^{(m)},
             \end{array}
\right.
\end{equation}
where
\begin{equation}\label{eq2.11}
  w^{(i)}=a_{i+1}-\frac{\langle a_{i+1},a_{i}\rangle}{\langle a_{i},a_{i}\rangle}a_{i},\,\, i=1,\cdots,m-1, \
  w^{(m)}=a_1-\frac{\langle a_1,a_{m}\rangle}{\langle a_{1},a_{m}\rangle}a_{m}.
\end{equation}
For convenience, we denote $a_ {m+1}\equiv a_ 1$, $b_ {m+1}\equiv b_ 1$. Then, for an arbitrary $p\ge 0$ and a given approximation $x^{(p,m)}\in R^{n}$ we construct the new ones $x^{(p+1,1)}$, $x^{(p+1,2)},\cdots ,x^{(p+1,m)}\in R^{n}$ by
\begin{equation}\label{eq2.12}
\left\{
             \begin{array}{lr}
             for \,\,\,\,i=1:m &  \\
             x^{(p+1,i)}=x^{(p+1,i-1)}+\frac{b_{i+1}-\langle a_{i+1},x^{(p+1,i-1)}\rangle}{\|w^{(i)}\|^2}w^{(i)},&\\
             end
             \end{array}
\right.
\end{equation}
with the notational convention
\begin{equation}\label{eq2.13}
  x^{(p+1,0)}=x^{(p,m)}.
\end{equation}
Obviously, $x^{(k+1)}=x^{(p,i)}$, if $k=p\cdot m+i, \ 0\le i<m$. The convergence of the KO method is provided as follows.
\begin{thm}
Let $x^{(0)}\in R^n$ be an arbitrary initial approximation, $\tilde{x}$ is a solution of \eqref{1.1}
such that $P_{N(A)}(\tilde{x})= P_{N(A)}(x^{(0)})$, and the sequence $\{x^{(k)}\}_{k=1}^{\infty}$ is generated with the KO
algorithm. Then,
\begin{equation}\label{eq2.14}
\lim\limits_{k\rightarrow\infty}x^{(k)}=\tilde{x}.\end{equation}
In addition, if $x^{(0)}\in R(A^T)$, then $\{x^{(k)}\}$ converges to
the least-norm solution of \eqref{1.1}, i.e.,
$$\lim\limits_{k\rightarrow\infty}x^{(k)}=x^*.$$
\end{thm}
\begin{proof}
According to \eqref{eq2.10}-\eqref{eq2.12} we obtain the sequence of approximations (from top to bottom and left to right, and by also using the notational convention \eqref{eq2.13})
\begin{equation}\label{eq2.15}
\left\{
             \begin{array}{lr}
             x^{(0)},x^{(0,0)}&  \\
             x^{(0,1)},x^{(0,2)},\cdots,x^{(0,m)}=x^{(1,0)} &  \\
             x^{(1,1)},x^{(1,2)},\cdots,x^{(1,m)}=x^{(2,0)} &  \\
             \cdots \cdots \cdots \cdots \cdots \cdots \cdots \cdots \cdots&\\
             x^{(p,1)},x^{(p,2)},\cdots,x^{(p,m)}=x^{(p+1,0)} &  \\
             \cdots \cdots \cdots \cdots \cdots \cdots \cdots \cdots \cdots
             \end{array}
\right.
\end{equation}
We define the numbers
\begin{equation}\label{eq2.16}
  r^{(p,i)}=b_{i+1}-\langle a_{i+1},x^{(p,i-1)} \rangle, \ i=1,2,\cdots,m, \ \forall p\ge 0.
\end{equation}
By using \eqref{eq2.10}-\eqref{eq2.16}, we can obtain
\begin{equation}\label{eq2.17}
  x^{(p,i)}=x^{(p,i-1)}+\frac{r^{(p,i)}}{\|w^{(i)}\|^2}w^{(i)},\,\,\,p\ge 0, \ i=1,2,\cdots,m.
\end{equation}
With Lemma \ref{lem2.4}, we have
\begin{equation}\label{eq2.18}
  \langle  x^{(p,i-1)}- x^{(p,i)}, x^{(p,i)}-\tilde{x}\rangle=0.
\end{equation}
Therefore, from $x^{(p,i-1)}- \tilde{x}=x^{(p,i-1)}-x^{(p,i)}+ x^{(p,i)}-\tilde{x}$, it is easy  to see
$$ \|x^{(p,i-1)}-\tilde{x}\|^2=\|x^{(p,i-1)}-x^{(p,i)}\|^2+\|x^{(p,i)}-\tilde{x}\|^2.$$
From \eqref{eq2.17}, we get
\begin{equation}\label{eq2.19}
  \|x^{(p,i-1)}-\tilde{x}\|^2=\|x^{(p,i)}-\tilde{x}\|^2+\frac{|r^{(p,i)}|^2}{\|w^{(i)}\|^2}.
\end{equation}
Obviously, the sequence $\{\|x^{(p,i)}-\tilde{x}\|\}_{p=0, i=0}^{\infty, m}$, i.e., $\{\|x^{(k+1)}-\tilde{x}\|\}_{k=1}^{\infty}$ is a monotonically decreasing sequence with lower bounds. There exists a $\alpha \ge 0 $ such that
\begin{equation}\label{eq2.20}
  \lim_{p\rightarrow\infty}\|x^{(p,i)}-\tilde{x}\|=\alpha\ge 0, \ \ \forall \ i=0,1,\cdots,m-1.
\end{equation}
Thus, from \eqref{eq2.19} and because $i$ was arbitrary we get
\begin{equation}\label{eq2.21}
  \lim_{p\rightarrow\infty}r^{(p,i)}= 0,\ \ \forall \ i=0,1,\cdots,m-1.
\end{equation}
Because the sequence $\{\|x^{(p,i)}-\tilde{x}\|\}_{p=0, i=0}^{\infty, m-1}$ is bounded, we obtain
\begin{equation}\label{eq2.22}
  \|x^{(p,i)}\|\le \|\tilde{x}\|+\|x^{(p,i)}-\tilde{x}\|\le \|\tilde{x}\|+\|x^{(0,1)}-\tilde{x}\|,\,\, \forall p\ge 0.
\end{equation}
According to the convention \eqref{eq2.22} we get that the sequence $\{x^{(p,0)}\}_{p=0}^{\infty}$ is bounded, thus there exists a convergent subsequence $\{x^{(p_j, 0)}\}_{j=1}^{\infty}$, let's denote it as
\begin{equation}\label{eq2.23}
  \lim_{j\rightarrow\infty}x^{(p_j, 0)}=\hat{x}.
\end{equation}
But, from \eqref{eq2.17} we get
\begin{equation}\label{eq2.24}
  x^{(p_j,1)}=x^{(p_j,0)}-\frac{r^{(p_j,1)}}{\|w^{(2)}\|^2}w^{(2)},\,\,\,\forall \ j>0.
\end{equation}
thus, by taking the limit following $j$ and using \eqref{eq2.21}, \eqref{eq2.23}
\begin{equation}\label{eq2.25}
  \lim_{j\rightarrow\infty}x^{(p_j, 1)}=\hat{x}.
\end{equation}
With the same way we obtain
\begin{equation}\label{eq2.26}
  \lim_{j\rightarrow\infty}x^{(p_j,i)}=\hat{x},\,\,\,\,\forall \ i=0, 1,\cdots, m-1.
\end{equation}
Then, from \eqref{eq2.26} we get for any $i=1, \cdots, m$
\begin{equation}\label{eq2.27}
  \lim_{j\rightarrow\infty}\langle x^{(p_j,i-1)},a_{i+1}\rangle-b_{i+1}=\langle\hat{x},a_{i+1}\rangle-b_{i+1},
\end{equation}
and from \eqref{eq2.16} and \eqref{eq2.21}
\begin{equation}\label{eq2.28}
  \lim_{j\rightarrow\infty}\langle x^{(p_j,i-1)},a_{i+1}\rangle-b_{i+1}=0, \ \ i=1, \ \cdots, \ m.
\end{equation}
Thus, from \eqref{eq2.27}$-$\eqref{eq2.28} it results in
\begin{equation}\label{eq2.29}
  \langle\hat{x},a_{i+1}\rangle-b_{i+1}=0,\,\,\,\forall \ i=1,\cdots,m,
\end{equation}
that is
\begin{equation}\label{eq2.30}
  A\hat{x}=b.
\end{equation}
With the use of the iterative relations
$$x^{(0,1)}=x^{(0,0)}+\frac{b_1-\langle a_1,x^{(0,0)}\rangle}{\|w^{(1)}\|^2}w^{(1)},$$
and
$$x^{(p,i)}=x^{(p,i-1)}+\frac{r^{(p,i)}}{\|w^{(i)}\|^2}w^{(i)},$$
$w^{(i)}$ and $w^{(1)}$ are defined in \eqref{eq2.11}. It is easy to deduce that $P_{N(A)}(x^{(k)})= P_{N(A)}(x^{(0)})$, and so
\begin{equation}\label{eq2.31}
P_{N(A)}(\hat{x})= P_{N(A)}(x^{(0)}).
\end{equation}
From the hypothesis of the theorem, we know that
\begin{equation}\label{eq2.32}
A\tilde{x}=b, \ \ P_{N(A)}(\tilde{x})=P_{N(A)}(x^{(0)}).
\end{equation}
By \eqref{eq2.30}$-$\eqref{eq2.32}, we get
$$\lim\limits_{j\rightarrow\infty}x^{(p_j,i)}-\hat{x}=\lim\limits_{j\rightarrow\infty}x^{(p_j,i)}-\tilde{x}=0, \ \forall \ i=0,1,\cdots,m-1.$$
If we set $k_j=p_j\cdot m+i$, then $\lim\limits_{j\rightarrow\infty}\|x^{(k_j)}-\tilde{x}\|=\alpha=0$. Based on monotonicity, $\lim\limits_{k\rightarrow\infty}\|x^{(k)}-\tilde{x}\|=0$, so the sequence $\{x^{(k)}\}$ is convergent to $\tilde{x}$.

In addition, if $x^{(0)}\in R(A^T)$, then $P_{N(A)}(x^{(0)})=0$ and so $\{x^{(k)}\}$ converges to
the least-norm solution of \eqref{1.1}, i.e.,
$$\lim\limits_{k\rightarrow\infty}x^{(k)}=x^*.$$

\end{proof}

{\bf Remark 1.} \ For the Kaczmarz method, it holds $\|x^{(k+1)}-\tilde{x}\|^2=\|x^{(k)}-\tilde{x}\|^2-\frac{{r_{i_{k+1}}^{(k)}}^2}{\|a_{i_{k+1}}\|^2}$, and the KO method holds $\|x^{(k+1)}-\tilde{x}\|^2=\|x^{(k)}-\tilde{x}\|^2-\frac{{r_{i_{k+1}}^{(k)}}^2}{\|a_{i_{k+1}}\|^2}\frac{1}{sin^2(\theta_{i_{k+1}})}$. So the KO method is faster than the Kaczmarz method if $0<\theta_{i_{k+1}}<\pi/2$.

{\bf Remark 2.} \ When the coefficient matrix $A$ is a matrix with orthogonal rows, the KO algorithm degenerates to the Kaczmarz algorithm (right now, $y^{(k)}=z^{(k)}=x^{(k+1)}$).
Generally, we use the KO method in two ways: one is online mode, and the other is preprocessing mode.

(1) Online mode. Each iteration only uses all the information of two adjacent equations, and there is no preprocessing information. Considering that the information of one equation is shared by two adjacent iterations, the KO method takes about $10n+2$ flops per iteration step. In this case, the Kaczmarz method algorithm needs $6n-1$ floating-point operations per step.

 (2) Preprocessing mode. Because the norm of row vector of matrix $A$, the inner product of two adjacent row vectors, the direction $w^{(i_k)}$ and its norm are fixed, these can be calculated in advance. After preprocessing (assuming that the above quantities have been calculated), the amount of floating-point number operation of the KO method in each step is $4n+1$ (only $r_{i_{k+1}}^{(k)}$ and $\alpha_{i_k}^{(k)}$ need to be calculated).  In this case, the workload of the KO method per iterative step is the same as that of the Kaczmarz method. But the total cost of pretreatment for the KO method is about $6mn+m$, while that for the Kaczmarz method is $2mn-1$.

 {\bf Remark 3.} \ Although the workload of the KO method in each step is more than or equals to that of the Kaczmarz method, compared with the Kaczmarz method, the iteration steps of the KO method are significantly reduced, especially for those problems which contain some linear equations with near linear correlation. See Example 2.1 and numerical experiments in Section 4.

\begin{exm}\label{EX2.1}\upshape
Consider the following systems of linear equations with two equations
\begin{equation}\label{Ex1.1}
\left\{
\begin{array}{rcl}
7x_1-8x_2&=&-1,\\
8x_1-7x_2&=&1
\end{array}
\right.
\end{equation}
and
\begin{equation}\label{Ex1.2}
\left\{
\begin{array}{rrl}
7x_1+8x_2&=&15,\\
140x_1+159x_2&=&299.
\end{array}
\right.
\end{equation}
\end{exm}
The two equations in system \eqref{Ex1.2} are close to correlation. So if the Kaczmarz method is used, $817$ steps are needed for the system (2.33) and $940,627$ steps are needed for the system (2.34) to reach the error requirement $\|x^{(k)}-x^*\|\leq\frac{1}{2}\times 10^{-6}$; but with the use of the KO method, both systems need only one step to get the exact solutions.

\section{Randomized Kaczmarz Method with Oblique Projection}
If the row index $i_{k+1}$ in Algorithm 2.1 is randomly selected, we get a randomized Kaczmarz method with oblique projection and its convergence as follows. Based on the relationship of the KO and the RKO methods, we can easily prove the expected convergence rate of the RKO method.
\begin{algorithm}
  \leftline{\caption{Randomized Kaczmarz Method with Oblique Projection (RKO)}}
  \label{alg:Framwork1}
  \begin{algorithmic}[1]
    \Require
      $A\in R^{m\times n}$, $b\in R^{m}$, $x^{(0)}\in R^n$, $K$, $\varepsilon$
    \State For $i=1:m$, $M(i)=\|a_i\|^2$
    \State Randomly select $i_1$, and compute $x^{(1)}=x^{(0)}+\frac{b_{i_1}-\langle a_{i_1}, x^{(0)}\rangle}{M(i_1)}a_{i_1}$
    \State Randomly select $i_2\not=i_{1}$, and compute $w^{(i_1)}=a_{i_2}-\frac{\langle a_{i_2}, a_{i_1}\rangle}{\|a_{i_1}\|^2}a_{i_1}$, $x^{(2)}=x^{(1)}+\frac{b_{i_2}-\langle a_{i_2}, x^{(1)}\rangle}{\|w^{(i_1)}\|^2}w^{(i_1)}$
    \For {$k=2,3,\cdots, K$}
    \State Randomly select $i_{k+1}$ ($i_{k+1}\not=i_{k}, \ i_{k-1}$) \ (uniformly at random)
    \State Compute $D_{i_k}=\langle a_{i_{k}}, a_{i_{k+1}}\rangle$, \ $r_{i_{k+1}}^{(k)}=b_{i_{k+1}}-\langle a_{i_{k+1}}, x^{(k)}\rangle$
    \State Compute $w^{(i_k)}=a_{i_{k+1}}-\frac{D_{i_k}}{M(i_{k})}a_{i_{k}}$, $h_{i_k}=M(i_{k+1})-\frac{D_{i_k}}{M(i_{k})}D_{i_k}$ ($=\|w^{(i_k)}\|^2$)
    \If {$h_{i_k}>\varepsilon$}
    \State $\alpha_{i_k}^{(k)}=\frac{r_{i_{k+1}}^{(k)}}{h_{i_k}}$ \ and $x^{(k+1)}=x^{(k)}+\alpha_{i_k}^{(k)} w^{(i_k)}$
    \EndIf
    \EndFor
    \State Output $x^{(k+1)}$
  \end{algorithmic}
\end{algorithm}

\begin{lem}
  Let $x^{(0)}\in R(A^T)$ be an arbitrary initial approximation, $x^*$ is the least-norm solution of \eqref{1.1} ($m>2$). We select $i_{k+1}\not=i_{k}, i_{k-1}$ uniformly at random and compute the next iteration $x^{(k+1)}=x^{(k)}+r_{i_{k+1}}^{(k)}\frac{w^{(i_k)}}{\|w^{(i_k)}\|^2}$, then we obtain the bound on the following expected conditional on the first $k$ ($k\ge 2$) iterations of the RKO method
$$E_k\frac{{r_{i_{k+1}}^{(k)}}^2}{\|w^{(i_k)}\|^2}\geq\frac{\sigma^2_{min}\|x^{(k)}-x^*\|^2}{(m-2)(\|A\|_F^2-\sigma^2_{min})}.$$
\end{lem}

\begin{proof}
 Due to $x^{(0)}\in R(A^T)$ and $x^*$ is the least-norm solution of equations \eqref{1.1}, $x^{(k)}-x^*\in R(A^T)$.
\begin{equation}
\begin{array}{rl}
E_k\frac{{r_{i_{k+1}}^{(k)}}^2}{\|w^{(i_k)}\|^2}=&\frac{1}{m-2}\sum\limits_{s=1, s\not=i_{k}, i_{k-1}}^m\frac{|r_{s}^{(k)}|^2}{\|w_{s}\|^2} \ (r_{i_{k}}^{(k)}=0 \ and \ r_{i_{k-1}}^{(k)}=0) \\
\geq&\frac{1}{m-2}\frac{\sum\limits_{s=1, s\not=i_{k}, i_{k-1}}^m|r_{s}^{(k)}|^2}{\sum\limits_{s=1, s\not=i_{k}, i_{k-1}}^m\|w_{s}\|^2}\\
=&
\frac{1}{m-2}\frac{\sum\limits_{s=1}^m|r_{s}^{(k)}|^2}{\sum\limits_{s=1}^m\|w_{s}\|^2}\\
=&\frac{1}{m-2}\frac{\|b-Ax^{(k)}\|^2}{\|A\|_F^2-\frac{\|Aa_{i_{k}}\|^2}{\|a_{i_{k}}\|^2}} \ (where \ b-Ax^{(k)}=A(x^*-x^{(k)}))\\
\geq&\frac{\sigma^2_{min}\|x^{(k)}-x^*\|^2}{(m-2)\left(\|A\|_F^2-\sigma^2_{min}\right)}.
 \end{array}
\end{equation}
Since $x^{(k)}$ is on the intersection of hyperplanes $\langle a_{i_{k}}, x\rangle=b_{i_{k}}$ and $\langle a_{i_{k-1}}, x\rangle=b_{i_{k-1}}$, we have $r_{i_{k}}^{(k)}=0 \ and \ r_{i_{k-1}}^{(k)}=0$. Thus the first and second equalities are valid. With $\sum\limits_{s=1}^m\|w_{s}\|^2=\|A\|_F^2-\sigma^2_{min}$, the last equality holds. The first inequality uses the conclusion of $\frac{|b_1|}{|a_1|}+\frac{|b_2|}{|a_2|}\geq\frac{|b_1|+|b_2|}{|a_1|+|a_2|}$ (if $|a_1|>0$, $|a_2|>0$), and the second one uses the conclusion of $\|Az\|_2\geq\sigma_{min}(A)\|z\|_2$, if $z\in R(A^T)$.
\end{proof}

\begin{thm}
  Assume that the system \eqref{1.1} is consistent, $m>2$ and $x^{(0)}\in R(A^T)$. Then the RKO method converges to the least-norm solution of equations \eqref{1.1} in expectation and has the following bound
$$E_k\|x^{(k+1)}-x^*\|^2\leq\left(1-\frac{\sigma^2_{min}}{(m-2)(\|A\|_F^2-\sigma^2_{min})}\right)\|x^{(k)}-x^*\|^2, \ \ k\ge 2.$$
\end{thm}

\begin{proof}
 From lemma \ref{lem2.4} ( see the description of Fig. 2.1 and Fig. 2.2), we know that
$$\langle x^{(k+1)}-x^{(k)}, x^{(k+1)}-x^{*}\rangle=0.$$
Therefore,
$$\|x^{(k)}-x^{*}\|^2=\|(x^{(k)}-x^{(k+1)})+(x^{(k+1)}-x^{*})\|^2=\|x^{(k)}-x^{(k+1)}\|^2+\|x^{(k+1)}-x^{*}\|^2.$$
By
$$\|x^{(k)}-x^{(k+1)}\|^2=\|r_{i_{k+1}}^{(k)}\frac{w^{(i_k)}}{\|w^{(i_k)}\|^2}\|^2=\frac{|r_{i_{k+1}}^{(k)}|^2}{\|w^{(i_k)}\|^2},$$
we know that
\begin{equation}
\begin{array}{rl}
E_k\|x^{(k+1)}-x^*\|^2&=\|x^{(k)}-x^{*}\|^2-E_k\frac{|r_{i_{k+1}}^{(k)}|^2}{\|w^{(i_k)}\|^2}\\
 &\leq\|x^{(k)}-x^{*}\|^2-\frac{\sigma^2_{min}}{(m-2)(\|A\|_F^2-\sigma^2_{min})}\|x^{(k)}-x^*\|^2 \ \ (Lemma \ 2.3)\\
 &\leq\left(1-\frac{\sigma^2_{min}}{(m-2)(\|A\|_F^2-\sigma^2_{min})}\right)\|x^{(k)}-x^*\|^2.
 \end{array}
\end{equation}
It yields the desired results.
\end{proof}

{\bf Remark 4.} From Theorem 3.1, we see that the convergence rate of the RKO method is faster than that of the RK method.

\section{Numerical Experiments}
In this section, we will present some experiment results of the Kaczmarz (K) method, randomized Kaczmarz (RK) method (with uniform probability), the Kaczmarz method with oblique projection (KO) and randomized Kaczmarz method with oblique projection (RKO) for solving the consistent linear system \eqref{1.1} with the coefficient matrix $A \in R^{m\times n}$ from three sources: Gaussian matrix, some real world matrices and {\bf Sprand} matrix.

In our implementations, the right vector $b\in R^n$ is chosen such that the exact solution $x^*\in R^n$ is a vector with all $1's$.
Define the relative solution error (RSE) at the $k$th iteration as follows:
$$\text{RSE}=\frac{\|x^{(k)}-x^*\|^2}{\|x^*\|^2}.$$
The initial point $x^{(0)}\in R^n$ is set to be a zero vector, and the iterations are terminated once the relative solution error satisfies $\text{RSE}<0.5\times 10^{-6}$ or the number of iteration steps exceeds 100,000. If the number of iteration steps exceeds 100,000, it is denoted as ``-".

We will compare the numerical performance of these methods in terms of the number of iteration steps (denoted as ``IT") and the computing time in seconds (denoted as ``CPU(s)"). Here the CPU(s) and IT mean the arithmetical averages of the elapsed running
times and the required iteration steps with respect to 50 trials repeated runs of the corresponding method.

All experiments are carried out by using MATLAB (version R2017b) on a DESKTOP-8CBRR86
with Intel(R) Core(TM) i7-4790, CPU 3.60GHz, RAM 8GB and Windows 10.

\begin{exm}\label{EX4.1}  \upshape Gaussian matrix. The Gaussian matrix is randomly generated by using the MATLAB function rand.
Consider the linear system \eqref{1.1} with $A=rand(m,500)$. The numerical results are reported in Table \ref{tab1}. From the table, we can conclude some observations as follows. First, the KO and RKO methods outperform the K and RK methods in terms of the iteration step. The number of iteration steps of the latter two is approximately twice that of the former two respectively. Second, we see that the KO method has an advantage over the K method in CPU. Third, the RKO method requires almost the same iteration steps as the KO method and consumes much less than the other three methods. Finally, we observe that under the same conditions (the same starting vector, the number of rows $m$ and termination condition) the execution time by the KO and RKO methods are quite less than that by the Kaczmarz method and the RK method, respectively.

\begin{table}

\centering
\begin{tabular}{ ccccccccc }
\hline
 Method & \multicolumn{2}{c}{K} & \multicolumn{2}{c}{KO} & \multicolumn{2}{c}{RK} & \multicolumn{2}{c}{RKO }  \\
\cline{2-9}
$m\times n$ &           IT&CPU&            IT&CPU &               IT&CPU&             IT&CPU\\
\hline
$1000\times 500$ &    81858&     0.4886 &  23895& 0.2485&67857&  5.8256	&27482& 0.7393\\
$2000\times 500$ &    32535&	0.2008&	9657&	0.1013&	30872&	1.7679&	12224&	 0.3742\\
$3000\times 500$ &    24219&	0.1516&	8488&	0.1239&	24504&	1.3120&	10890&	 0.3843\\
$4000\times 500$ &    21224&	0.1390&	8138&	0.1060&	23645&	1.3236&	9830&	 0.3587\\
$5000\times 500$ &    19360&	0.1357&	8381&	0.1169&	21313&	1.1556&	9940&	 0.3682\\
$6000\times 500$ &    18889&	0.1355&	8250&	0.1117&	20319&	1.0759&	9653&	 0.3907\\
$7000\times 500$ &    17529&	0.1367&	8263&	0.1388&	20129&	1.0678&	9054&	 0.3623\\
$8000\times 500$ &    17406&	0.1562&	8355&	0.1326&	18108&	0.9087&	9779&	 0.4008\\
$9000\times 500$ &    16959&	0.1315&	8414&	0.1262&	17575&	0.8727&	9686&	 0.4167\\
$10000\times 500$&    16658&	0.1234&	8151&	0.1198&	18122&	0.8863&	9579&	 0.4064\\
\hline
\end{tabular}
\caption{IT and CPU of K, KO, RK and RKO for m-by-n matrices A with n=500 and different m when the linear system is consistent.}
\label{tab1}
\end{table}
\end{exm}

\begin{exm}\label{EX4.2}\upshape Real world matrix.
The real world sparse matrices are taken from \cite{DH11}, which include well-conditioned matrices and ill-conditioned matrices. The properties of different sparse matrices are shown in the Table \ref{tab2} We list the numbers of IT and the CPU for the four methods in Table \ref{tab3}. The results show that the KO and RKO methods can always successfully compute an approximate solution to the linear system \eqref{1.1}, but the K and RK methods fail for the matrices $WorldCities$ and $well1033$ due to the numbers of the iteration steps exceeding 100,000. For all convergent cases, IT and CPU of the KO method are considerably smaller than those of the K method. In the meantime, the RKO method significantly outperforms the RK method in terms of both IT and CPU, too.
\end{exm}

\begin{table}

\centering
\begin{tabular}{ ccccccc }
\hline
\multicolumn{2}{c}{ name  }           &Stranke94    & Trefethen$_{-}$20  & ash608 & WorldCities    & well1033\\
\hline
 \multicolumn{2}{c}{ $m\times n$  }   &$10\times 10 $ & $20\times20$    & $608\times188$       & $315\times100$  & $1033\times320$  \\

\multicolumn{2}{c}{ density   }       &$90.00\%$   & $39.50\%$  & $1.06\%$  & $23.87\%$ & $1.43\%$\\

\multicolumn{2}{c}{ cond(A)  }        & 51.73 &  63.09 &  3.37 & 66.00 &166.13 \\
\hline
\end{tabular}
\caption{The properties of different sparse matrices.}
\label{tab2}
\end{table}

\begin{table}

\centering
\begin{tabular}{ ccccccccc }
\hline
 Method & \multicolumn{2}{c}{K} & \multicolumn{2}{c}{KO} & \multicolumn{2}{c}{RK} & \multicolumn{2}{c}{RKO }  \\
\cline{2-9}
name &           IT&CPU&            IT&CPU &               IT&CPU&             IT&CPU\\
\hline
 Stranke94         &   5878&    0.0707 & 3846	& 0.0446 &14020	& 0.4748 &3517& 0.0633\\
 Trefethen$_{-}$20 &   201&	     0.0012&	  111&	 0.0008&		1186	&	 0.0157	&	742&	0.0114\\
 ash608          &    2652	    &0.0219	&	1705	&	0.0186	&	1211	&	 0.0194	&	998&	0.0211\\
 WorldCities     &   -&	-&  	20317&		0.6395&	-&	-&22257&0.8953\\
 well1033       &-&-  &66079&	5.1150	&-&-&168887&	28.9278\\
\hline
\end{tabular}
\caption{IT and CPU of K, KO, RK and RKO for m-by-n matrices A with different m and n.}
\label{tab3}
\end{table}

\begin{exm}\label{EX4.3}\upshape Uniformly distributed matrix on $[c,1]$. Consider the consistent linear system \eqref{1.1} with uniformly distributed coefficient matrix $A\in R^{m\times m}$ on $[c,1]$, which is generated from the MATLAB function sprand. We perform several experiments to compare IT and CPU of the four methods. All methods are run with the same fixed initial (zero vector) estimate and the fixed matrix. The numerical results are reported in Table \ref{tab4}.

\begin{table}

\centering
\begin{tabular}{ ccccccccc }
\hline
Method & \multicolumn{2}{c}{K} & \multicolumn{2}{c}{KO} & \multicolumn{2}{c}{RK} & \multicolumn{2}{c}{RKO }  \\
\cline{2-9}
$m=10^{4},n=500$ &           IT&CPU&            IT&CPU &               IT&CPU&             IT&CPU\\
\hline
$c=0.15$ & 25193&	0.1777&	8281&	0.1317&	27007.6&	1.6106&	  8544.9&	0.3826\\
$c=0.25$&  36825&	0.2529&	8014&	0.1266&	37109.8&	2.4720&	   8398.3&	0.3760\\
$c=0.5$&   -&	-&	7069&	0.1163&	96565&	    11.0027&	7680.4&	0.3427\\
$c=0.75$&  -&	-&	6477&	0.1080&	-&	    -&	6431.8&	0.2924\\
$c=0.9$&   -&	-&	4953&	0.0907&	-&	    -&	5011.1&	0.2462\\
\hline
\end{tabular}
\caption{$A\in R ^{10000\times 500}$ uniformly distributed on $[c,1]$.}
\label{tab4}
\end{table}

\begin{table}

\centering
\begin{tabular}{ ccccccccc }
\hline
 Method & \multicolumn{2}{c}{K} & \multicolumn{2}{c}{KO} & \multicolumn{2}{c}{RK} & \multicolumn{2}{c}{RKO }  \\
\cline{2-9}
m&           IT&CPU&            IT&CPU &               IT&CPU&             IT&CPU\\
\hline
$m=2000$&	34125&	0.1982&	10414&	0.1247  &31238.3&  1.8202	&11432	&0.3663\\
$m=4000$&	21480&	0.1364	&8412&	0.1109	&22301.3&  1.1980	&9293.2	&0.3618\\
$m=6000$&	18425&	0.1369	&8485&	0.1191	&20010.8&  1.0364	&8960.6	&0.3649\\
$m=8000$&	16966&	0.1422	&8114&	0.1318	&19509&    0.9777   &8913	&0.3901\\
$m=10000$&	16302&	0.1267	&8347&  0.1232	&18599&	   0.9458	&8741	&0.3746\\
\hline
\end{tabular}
\caption{$A$ uniformly distributed on $[0,1]$.}
\label{tab5}
\end{table}

Table \ref{tab4} shows that the KO and RKO methods can always successfully solve the problem \eqref{1.1}. The K and RK methods cannot obtain a solution when c is close to 1, because the number of iteration steps exceeds 100,000. Besides, when all four methods converge, the KO and RKO methods are significantly better than the K and RK methods in terms of iteration step and CPU, respectively. It can be seen from Table \ref{tab5} that the KO method is significantly better than the K method in the number of iteration steps, and the KO method also requires less CPU to achieve convergence. Also, the RKO method is significantly better than the RK method in terms of iterative steps and CPU. In fact, the CPU required by the RK method is about three times that of the RKO method.

\end{exm}

\section{Conclusions}

Based on the oblique projections to the hyperplanes, we derive a new extension of the Kaczmarz method. The non-random KO method greatly improves the Kaczmarz method. Compared with the RK method, the Randomized version of oblique projection (RKO) can greatly reduce the number of iterations and running time for solving large-scale overdetermined consistent systems of equations $Ax=b$, especially for the uniformly distributed random data $A$ and $b$. Numerical experiments show the effectiveness of the two methods for uniformly distributed random data.

\bibliographystyle{model1-num-names}
\bibliography{datako}

\end{document}